\documentclass[english,11pt,a4paper,leqno]{amsart}
\usepackage{amsmath,amstext, amsthm, amssymb}
\usepackage{latexsym, amsthm, amsfonts, amsmath,amssymb,graphicx,xcolor}
\usepackage[latin1]{inputenc} 
\usepackage[T1]{fontenc} 
\usepackage{enumerate}
\usepackage[english]{babel} 
\usepackage{enumitem}
\usepackage{upgreek}
\usepackage[colorlinks=true, linkcolor=blue, hyperfootnotes=true,citecolor=blue,urlcolor=black]{hyperref}

\newtheorem{thm}{Theorem}

\newtheorem{prop}[thm]{Proposition}
\newtheorem{lem}[thm]{Lemma}

\theoremstyle{definition}

\theoremstyle{remark} 

\theoremstyle{definition}

\newcounter{constant}
\newcommand{\newconstant}[1]{\refstepcounter{constant}\label{#1}}
\newcommand{\useconstant}[1]{c_{\ref{#1}}}
\newcommand{\defconstant}[1]{ \newconstant{c_{#1}}\expandafter\newcommand\csname c#1\endcsname{\useconstant{c_{#1}}} }

\newcommand{\Q}{\overline{\mathbb Q}}

\newcommand{\N}{\mathbb{N}}

\newcommand{\M}{\mathcal M}

\newcommand{\I}{\mathcal I}

\newcommand{\z}{{\boldsymbol{z}}}
\newcommand{\lambd}{{\boldsymbol{\lambda}}}

\newcommand{\bmu}{{\boldsymbol{\mu}}}

\newcommand{\gamm}{{\boldsymbol{\gamma}}}

\newcommand{\kapp}{{\boldsymbol{\kappa}}}

\newcommand{\bbeta}{{\boldsymbol{\beta}}}

\numberwithin{equation}{section} 

\title{Addendum to: Mahler's method in several variables and finite automata}

\author{Boris Adamczewski}
\address{
Univ Lyon, Universit\'e Claude Bernard Lyon 1\\
 CNRS UMR 5208, Institut Camille Jordan \\
 F-69622 Villeurbanne Cedex, France}
\email{boris.adamczewski@math.cnrs.fr}

\author{Colin Faverjon}
\address{
Univ Lyon, Universit\'e Claude Bernard Lyon 1\\
 CNRS UMR 5208, Institut Camille Jordan \\
 F-69622 Villeurbanne Cedex, France}
\email{colin.faverjon@riseup.net}
\date{}

\thanks{This project has received funding from the European Research Council (ERC) under the 
European Union's Horizon 2020 research and innovation programme 
under the Grant Agreement No 648132. }

\begin{document}

\maketitle

The aim of this note is to prove the following extension of one of the main results of \cite{AF24_Annals} concerning the algebraic independence of values of 
$M$-functions at multiplicatively independent algebraic points. We retain the notations introduced in  \cite{AF24_Annals}.

\begin{thm}\label{thm: mainbis}
	Let $r\geq 1$ be an integer and $\mathbb K\subseteq \Q$ be a field. 
	For every integer $i$, $1\leq  i\leq r$, we let $q_i\geq 2$ be an integer, 
	$f_i(z)\in \mathbb K[[z]]$ be an $M_{q_i}$-function, 
	and $\alpha_i\in \mathbb K$, $0 <\vert\alpha_i\vert <1$, be such that $f_i(z)$ 
	is well-defined at $\alpha_i$. Let us assume that the numbers $\alpha_1,\ldots, \alpha_r$ are pairwise multiplicatively independent. 
	Then  $f_1(\alpha_1),f_2(\alpha_2),\ldots,f_r(\alpha_r)$ are algebraically independent over $\Q$, unless  
	one of them belongs to $\mathbb K$.   	
\end{thm}

Theorem~\ref{thm: mainbis} strengthens part (i) of \cite[Theorem~1.1]{AF24_Annals} in which a stronger condition was required: the points $\alpha_i$ had to be (globally) multiplicatively independent and not just pairwise multiplicatively independent.  For instance, assuming that $f_1(z),f_2(z)$ and $f_3(z)$ are $M$-functions that take transcendental values at $\frac{1}{2},\frac{1}{5}$ and $\frac{1}{10}$ respectively, 
Theorem~\ref{thm: mainbis} implies that these three numbers are algebraically independent, while  \cite[Theorem~1.1]{AF24_Annals}  could not apply. 

We deduce from Theorem \ref{thm: mainbis} the following generalization of \cite[Theorem~2.3]{AF24_Annals}.

\begin{thm}\label{conj: strongv_bis}
	Let $r\geq 1$ be an integer. Let $b_1,\ldots,b_r$ be pairwise multiplicatively independent positive integers, 
	and, for every $i$, $1\leq i \leq r$, 
	let  $x_i$ be a real number that is automatic in base $b_i$.   
	Then the numbers $x_1,\ldots,x_r$ are algebraically independent over $\Q$, unless one of them is rational. 
\end{thm}

We omit the proof of Theorem~\ref{conj: strongv_bis} as it can be deduced from Theorem~\ref{thm: mainbis}, just as \cite[Theorem 2.3]{AF24_Annals}  can be deduced from \cite[Theorem 1.1]{AF24_Annals}.

\medskip

The rest of this note is devoted to the proof of Theorem~\ref{thm: mainbis}.  As with the proof of \cite[Theorem 1.1]{AF24_Annals},  
it mainly relies on some of the general results concerning Mahler's method in several variables proved in \cite{AF24_Annals} (e.g., Corollary~3.5,  Corollary~3.9, and Theorem~5.9).  The main novelty is the use of a trick introduced by Loxton and van der Poorten~\cite{LvdP78}  in this framework to deal with values of  Mahler functions at certain points with multiplicatively dependent coordinates.

\section{Proof of Theorem  \ref{thm: mainbis}}

In order to prove Theorem  \ref{thm: mainbis}, we first need three auxiliary results. 

\subsection{Auxiliary results} 

Our first auxiliary result is a lemma concerning algebraic numbers, on which Loxton and van der Poorten's trick is based. 

\begin{lem}\label{prop:LvdP_mutl_indep_points}
	Let $\alpha_1,\ldots,\alpha_r \in \Q$ be algebraic numbers such that $0 < \vert \alpha_i \vert < 1$ for every $i$, $1\leq i\leq r$. Then there exist multiplicatively independent algebraic numbers $\beta_1,\ldots,\beta_t \in \Q$, $0 < \vert \beta_j \vert < 1$, $1 \leq j \leq t$, roots of unity $\zeta_1,\ldots,\zeta_r$, and nonnegative integers $\mu_{i,j}$, $1\leq i \leq r$, $ 1\leq j \leq t$, such that
	$$
	\alpha_i= \zeta_i \prod_{j=1}^t \beta_j^{\mu_{i,j}},\quad  \forall i,\, 1\leq i \leq r\,.
	$$ 
\end{lem}

\begin{proof}
	This is \cite[Lemma 3]{LvdP78} (see also \cite[Lemma 3.4.9]{Ni_Liv}).
\end{proof}

Our second auxiliary result is the following result about $M$-functions.

\begin{lem}\label{lem:powers}
	Let $q\geq 2$ be an integer, $f(z)$ be an $M_q$-function and $\zeta$ be a root of unity. Then $f(\zeta z)$ is also an $M_q$-function.
\end{lem}

\begin{proof}
   We first recall that the set of $M_q$-functions is a ring containing $\Q(z)\cap \Q[[z]]$ and that, given any  positive integer $\ell$, 
   a power series is an $M_q$-function if and only if it is an $M_{q^\ell}$-function.   
	Let $k$ be such that $\zeta_0:=\zeta^{q^k}$ has order coprime with $q$. Then there exists a positive integer $\ell$ such that $\zeta_0^{q^\ell}=\zeta_0$. 
	Since $f(z)$ is also an $M_{q^\ell}$-function, we deduce that $f(\zeta_0 z)$ is an $M_{q^\ell}$-function and hence an $M_q$-function. 
	The same argument applies to any power of $\zeta_0$, so that $f(\zeta_0^iz)$ is an $M_q$-function for every integer $i\geq 0$. Given a positive integer $j$,  
	substituting $z$ with $z^{q^{jk}}$ and taking $i:=q^{k(j-1)}$, we thus deduce that $f((\zeta z)^{q^{jk}})$ is an $M_q$-function. Now, substituting $\zeta z$ to $z$ in the minimal $q^k$-Mahler equation satisfied by $f(z)$, we can write $f(\zeta z)$ as a linear combination over $\Q(z)$ of  the series $f((\zeta z)^{q^{jk}})$, $j\in\{1,\ldots,r\}$, where $r$ is the order of this minimal equation. Since $f(\zeta z)$ is a power series, we can ensure that  $f(\zeta z)$ can in fact be written as a linear combination over $\Q(z)\cap \Q[[z]]$  of  some $M_q$-functions. It therefore follows that 
	$f(\zeta z)$ is an $M_q$-function, as  wanted.
\end{proof}

Our third auxiliary result is about algebraic independence of power series. 

\begin{lem}\label{lem:purity_functions}
	Let $r$ and $t$ be two positive  integers, $\bmu_1,\ldots,\bmu_r \in \N^t$ be  vectors that are pairwise linearly independent over $\mathbb Q$, and, for every $i$, $1\leq i \leq r$, let $m_i$ be a positive integer and $f_{i,1}(z),\ldots,f_{i,m_i}(z) \in \Q[[z]]$. Let $\z:=(z_1,\ldots,z_t)$ be a vector of indeterminates. 
	Then
	\begin{multline*}
	{\rm tr.deg}_{\Q(\z)}(f_{i,j}(\z^{\bmu_i})\,:\,1\leq i \leq r,\,1\leq j \leq m_i) \\= \sum_{i=1}^r	{\rm tr.deg}_{\Q(z)}(f_{i,j}(z)\,:\,1\leq j \leq m_i)\,.
	\end{multline*}
\end{lem}

We recall that $\z^{\bmu_j}:=\prod_{i=1}^t z_i^{\mu_{i,j}}$. In order to prove Lemma~\ref{lem:purity_functions}, we first need to establish a simple result about cones in $\mathbb R^t$. 
We define the convex cone $\mathcal C$ spanned by some vectors $\bmu_1,\ldots,\bmu_r \in \mathbb R^t$  as the set
$$
\mathcal C:=\{a_1\bmu_1+\cdots + a_r\bmu_r\ : \ a_1,\ldots,a_r \in \mathbb R_{\geq 0}\}\,.
$$
A basis of  $\mathcal C$ is a minimal set of vectors in $\mathbb R^t$ such that the convex cone spanned by these vectors is $\mathcal C$. 

\begin{lem}
	\label{lem:cones}
	Let $\bmu_1,\ldots,\bmu_r \in \N^t$ be pairwise linearly independent over $\mathbb Q$ and $\mathcal C$ denote the convex cone spanned by $\bmu_1,\ldots,\bmu_r$. Let us assume that $\{\bmu_1,\ldots,\bmu_s\}$ is a basis of $\mathcal C$, for some $1\leq s\leq r$. Then, $\bmu_1$ does not belong to the convex cone $\mathcal C^\circ$ spanned by  $\bmu_2,\ldots,\bmu_r$. Furthermore, for any $\lambd \in \N^t$ and any finite set $\Gamma \subset \N^t$, the intersection
	$$
	(\lambd + \N\bmu_1) \bigcap (\Gamma + \mathcal C^\circ)
	$$
	is finite.
\end{lem}

\begin{proof}
	Let us start with the first part of the proof. We first note that, since the vector $\bmu_1,\ldots,\bmu_r$ are pairwise linearly independent over $\mathbb Q$, $\bmu_1$ is a nonzero vector. By assumption, for every $i$, $s<i\leq r$, there exist nonnegative real numbers $\lambda_{i,j}$, $1\leq j \leq s$, such that 
	\begin{equation}\label{eq: lambda}
	\bmu_i = \sum_{j=1}^s \lambda_{i,j} \bmu_j \,.
	\end{equation}
	 Let us assume by contradiction that $\bmu_1$ belongs to the convex cone spanned by $\bmu_2,\ldots,\bmu_r$. Then, there exist nonnegative real numbers $\theta_2,\ldots,\theta_r$ such that 
	 \begin{equation}\label{eq: theta}
	 \bmu_1 =\sum_{j = 2}^r \theta_j \bmu_j\,.
	 \end{equation}
	 We deduce from \eqref{eq: lambda} and \eqref{eq: theta} that 
		\begin{eqnarray*}
			\bmu_1 & = & \sum_{j=2}^s \theta_j \bmu_j + \sum_{i=s+1}^r \theta_i\sum_{j=1}^s \lambda_{i,j}\bmu_j
			\\ & = & \left(\sum_{i=s+1}^r \theta_i\lambda_{i,1}\right)\bmu_1 + \sum_{j=2}^s \left(\theta_j+\sum_{i=s+1}^r \theta_i\lambda_{i,j}\right)\bmu_j\, 
		\end{eqnarray*}
and hence 
	$$
	\left(1 - \sum_{i=s+1}^r \theta_i\lambda_{i,1}\right)\bmu_1 =  \sum_{i=2}^s \left(\theta_j+\sum_{i=s+1}^r \theta_i\lambda_{i,j}\right)\bmu_j \, .
	$$
	On the one hand, if $1 - \sum_{i=s+1}^r \theta_i\lambda_{i,1}> 0$, then $\bmu_1$ would belong to the convex cone generated by $\bmu_2,\ldots,\bmu_s$, which would contradict the fact that $\{\bmu_1,\ldots,\bmu_s\}$ is a basis of $\mathcal C$. On the other hand, if $1 - \sum_{i=s+1}^r \theta_i\lambda_{i,1}<0$, since $\bmu_1 \neq 0$, at least one of the coordinates of $\bmu_1$ would be negative, which is impossible. Hence $1 - \sum_{i=s+1}^r \theta_j\lambda_{i,1}=0$ and we deduce that
	\begin{equation}\label{eq: thetalambda}
	\theta_j+\sum_{i=s+1}^r \theta_i\lambda_{i,j} = 0,\quad \forall j,\, 2\leq j \leq s\,.
	\end{equation}
	Since all these numbers are nonnegative, we first observe that $\theta_j = 0$, for every $j\in\{2,\ldots, s\}$. Since $\bmu_1$ is nonzero, we infer from \eqref{eq: theta} the existence of $i_0>s$  such that $\theta_{i_0} \neq 0$. Then, we deduce from \eqref{eq: thetalambda} that $\lambda_{i_0,j}=0$ for every $j\in \{2,\ldots, s\}$. Thus, it follows from \eqref{eq: lambda} that 
	$\bmu_{i_0} = \lambda_{i_0,1} \bmu_1$, providing a contradiction with the fact that $\bmu_1,\ldots,\bmu_r$ are pairwise linearly independent over 
	$\mathbb Q$. This concludes the first part of the proof.
	
	\medskip We now turn to the second part. Let $\lambd\in\mathbb N^t$ and $\Gamma$ be a finite subset of $\mathbb N^t$. Let $d:=\inf_{\kapp \in \mathcal C^\circ} \vert \bmu_1 - \kapp \vert$ denote the distance between $\bmu_1$ and $\mathcal C^\circ$. Since we just proved that $\bmu_1$ does not belong to $\mathcal C^\circ$, we easily deduce that $d>0$. 
	Set
	$$
	B:= \max\{\vert \gamm\vert +\vert\lambd\vert : \gamm \in\Gamma\}\,.
	$$
	Let $k\in\mathbb N$ be such that $\lambd +k \bmu_1\in\Gamma+\mathcal C^\circ$. Then 
	$$
	\lambd+k\bmu_1=\gamm +\bmu \,,
	$$
	for some $\gamm \in\Gamma$ and $\bmu\in \mathcal C^\circ$. Since $\bmu/k\in\mathcal C^\circ$, it follows that 
	$$
	\frac{B}{k} \geq  \frac{\vert \gamm-\lambd\vert }{k} = \left\vert \bmu_1- \frac{\bmu}{k}\right\vert  \geq d 
	$$
	and hence $k\leq B/d$.  We deduce that  
	$$
	(\lambd +\N\bmu_1 )\cap (\Gamma+\mathcal C^\circ) \subset \{\lambd + k\bmu_1 \,:\, 0 \leq k \leq B/d\}\,.
	$$
	In particular, it is a finite set.
\end{proof}

\begin{proof}[Proof of Lemma \ref{lem:purity_functions}]
	We argue by induction on $r$. When $r=1$, there is nothing to prove. We now assume that $r\geq 2$ and that the result is proven for $r-1$. Up to reordering the indices, we can assume that $\{\bmu_1,\ldots,\bmu_s\}$ is a basis of the cone $\mathcal C$ spanned by $\bmu_1,\ldots,\bmu_r$, for some $s\leq r$. 
	According to our induction hypothesis, we only have to prove that
		\begin{multline*}
	{\rm tr.deg}_{\Q(\z)}(f_{i,j}(\z^{\bmu_i})\,:\,1\leq i \leq r,\,1\leq j \leq m_i)\\ ={\rm tr.deg}_{\Q(z)}(f_{1,j}(z)\,:\,1\leq j \leq m_1)\\+	{\rm tr.deg}_{\Q(\z)}(f_{i,j}(\z^{\bmu_i})\,:\,2\leq i \leq r,\,1\leq j \leq m_i)\,.
	\end{multline*}
	 We are going to prove the following stronger fact: for any $g_1(z),\ldots,g_m(z)\in \Q[[z]]$ that are linearly independent over $\Q(z)$, the power series $$g_1(\z^{\bmu_1}),\ldots,g_m(\z^{\bmu_1})\in\Q[[\z]]$$ are linearly independent over the ring $\mathbb A:=\Q[[\z^{\bmu_2},\ldots,\z^{\bmu_r}]][\z]$.
	
	\medskip Let $g_1(z),\ldots,g_m(z)\in \Q[[z]]$ be linearly independent over $\Q(z)$ and let us assume by contradiction that the series $g_1(\z^{\bmu_1}),\ldots,g_m(\z^{\bmu_1})$ are linearly dependent over $\mathbb A$. Then, there exist  $h_1(\z),\ldots,h_m(\z)\in \mathbb A$, not all zero, such that	
	\begin{equation}\label{eq:rel_lin_surcorps}
	h_1(\z)g_1(\z^{\bmu_1}) + \cdots + h_m(\z)g_m(\z^{\bmu_1})=0\,.
	\end{equation}
	Let $\mathcal C^\circ$ denote the convex cone spanned by $\bmu_2,\ldots,\bmu_r$. By definition of $\mathbb A$, there exists a finite set $\Gamma \subset \N^t$ such that the support of each $h_i(\z)$ is included in $\Gamma + \mathcal C^\circ$. Thus, we can write 
	$$
	h_i(\z) = \sum_{\kapp \in \Gamma+\mathcal C^\circ} h_{i,\kapp}\z^\kapp, \quad\quad \forall i, 1\leq i \leq m.
	$$
	We also set $h_{i,\kapp}:=0$ when $\kapp \notin \Gamma+\mathcal C^\circ$.
	Considering the equivalence relation on $\mathbb N^t$ defined by $\lambd_1 \sim \lambd_2$ if $\lambd_1-\lambd_2\in \mathbb Z \bmu_1$, we can defined a set 
	$\Lambda \subset \N^t$ by picking the vector of smallest norm in each equivalence class, so that 
	$\N^t$ can be written as the disjoint union $\bigsqcup_{\lambd\in \Lambda} \left(\lambd + \N\bmu_1\right)$. For every $\lambd \in \Lambda$, set $\Gamma_\lambd := (\Gamma+\mathcal C^\circ) \cap (\lambd + \N\bmu_1)$. It follows from Lemma~\ref{lem:cones} that all the sets $\Gamma_\lambd$ are finite. Since the sets $\Gamma_\lambd$, $\lambd \in \Lambda$, form a partition of $\Gamma+\mathcal C^\circ$, and since every element of $\Gamma_\lambd$ can be written $\lambd + n\bmu_1$ for some $n \in \N$,  we have a decomposition of the form 
	$$
	h_i(\z) = \sum_{\lambd \in \Lambda}\z^{\lambd}a_{i,\lambd}(\z^{\bmu_1}),
	 \quad \forall i, 1\leq i \leq m\,,
	$$
	where $a_{i,\lambd}(z):=\sum_{n =0}^\infty h_{i,\lambd + n \bmu_1}z^n$.	Since all the sets $\Gamma_\lambd$ are finite, the $a_{i,\lambd}(z)$ are in fact polynomials. Since the sets $\lambd + \N\bmu_1$, $\lambd\in\Lambda$,  are disjoints, identifying the powers of $\z$ that belong to $\lambd + \N\bmu_1$ in \eqref{eq:rel_lin_surcorps} leads to 
	$$
	\sum_{i=1}^m a_{i,\lambd}(z)g_i(z)=0\,, \quad\quad  \forall \lambd\in\Lambda \,.
	$$
	Since the power series $g_1(z),\ldots,g_m(z)$ are linearly independent over $\Q(z)$, we deduce that $a_{i,\lambd}(z)=0$ for every pair $(i,\lambd)\in\{1,\ldots,m\}\times \Lambda$. 
	Hence $h_i(\z)=0$ for all $i\in\{1,\ldots,m\}$, which provides a contradiction.
\end{proof}

\subsection{Existence of a suitable linear Mahler system}

The following proposition ensures the existence of suitable linear Mahler systems in several variables that will be used to deduce Theorem~\ref{thm: mainbis} from the main results of \cite{AF24_Annals}.

\begin{prop}
	\label{prop:existenceMpoints}
	Let $q\geq 2$ be an integer, $\alpha_1,\ldots,\alpha_r \in \Q$ be pairwise multiplicatively independent, $0 < \vert \alpha_i \vert < 1$ and, for every $i$, $1\leq i \leq r$, $f_i(z)\in\Q[[z]]$ be an $M_q$-function that is well defined at $\alpha_i$. Then there exist a positive integer $t$, a positive integer $\ell$, a point $\bbeta \in \Q^t$, a matrix $T \in \M_{t}(\N)$, some vectors $\bmu_1,\ldots,\bmu_r \in \N^t$, and for every $i$,  $1\leq i \leq r$, roots of unity $\zeta_i$, a positive integer $m_i$ and some $M_q$-functions $g_{i,1}(z),\ldots,g_{i,m_i}(z) \in \Q[[z]]$ such that the following hold. 
	
	\begin{itemize}
	
	\item[\rm (a)] For  every  $i\in\{1,\ldots,r\}$, $\alpha_i=\zeta_i\bbeta^{\bmu_i}$.
	
	\item[\rm (b)] For  every  $i\in\{1,\ldots,r\}$, $f_i(\alpha_i)=g_{i,1}(\bbeta^{\bmu_i})$.

	\item[\rm (c)] For  every  $i\in\{1,\ldots,r\}$,  $g_{i,1}(z),\ldots,g_{i,m_i}(z)$ are related by a $q^{\ell}$-Mahler system and $\bbeta^{\bmu_i}$ is regular 
	w.r.t.\ this system.

	\item[\rm (d)] The functions $g_{i,j}(\z^{\bmu_i})$, $1\leq i \leq r$, $1\leq j \leq m_i$ are related by a $T$-Mahler system, where $\z=(z_1,\ldots,z_t)$ is a vector of indeterminates.
	
	\item[\rm (e)] The pair $(T,\bbeta)$ is admissible and the point $\bbeta$ is regular w.r.t.\ this system.

		\item[\rm (f)] The spectral radius of $T$ is equal to $q^\ell$.
			
		\item[\rm (g)] The vectors $\bmu_1,\ldots,\bmu_r$ are pairwise linearly independent over $\mathbb Q$.

	\end{itemize}
\end{prop}

\begin{proof} 
	We first infer from Lemma~\ref{prop:LvdP_mutl_indep_points} the existence of a positive integer $t$, multiplicatively independent algebraic numbers $\beta_1,\ldots,\beta_t$, $0 < \vert \beta_j \vert < 1$, $1 \leq j \leq t$, roots of unity $\zeta_1,\ldots,\zeta_r$ and nonnegative integers $\mu_{i,j}$, $1\leq i \leq r$, $ 1\leq j \leq t$, such that
	$$
	\alpha_i= \zeta_i \prod_{j=1}^t \beta_j^{\mu_{i,j}},\quad\quad \forall i,\, 1\leq i \leq r\,.
	$$ 
	Setting $\bbeta:=(\beta_1,\ldots,\beta_t)$ and $\bmu_i:=(\mu_{i,1},\ldots,\mu_{i,t})$, we get that (a) is satisfied. 	
	By Lemma~\ref{lem:powers}, each $f_i(\zeta_i z)$ is an $M_q$-function. Applying \cite[Lemma~11.1]{AF24_Annals} to the functions $f_{i}(\zeta_i z)$ and the points $\zeta_i^{-1}\alpha_i=\bbeta^{\bmu_i}$, we can find, for every $i\in\{1,\ldots,r\}$, some $M_q$-functions $g_{i,1}(z),\ldots,g_{i,m_i}(z)$ related by some $q^{\ell_i}$-Mahler system with respect to which $\bbeta^{\bmu_i}$ is a regular point and such that $g_{i,1}(\bbeta^{\bmu_i})=f_i(\alpha_i)$, so that (b) holds. 
	Iterating each one of these systems an appropriate number of times if necessary, we can further assume that the integers $\ell_i$, $1\leq i \leq r$, are all equal to some common integer, say $\ell$. Hence (c) is satisfied. 	
	Let $A_1(z),\ldots,A_r(z)$ denote the matrices associated with each of these Mahler systems. 
		Let $\z:=(z_1,\ldots,z_t)$ be a vector of indeterminates and let $B(\z)$ denote the block-diagonal matrix with blocks
	$A_1(\z^{\bmu_1}),\ldots,A_r(\z^{\bmu_r})$. Set $T:=q^\ell{\rm I}_t$. By construction, the functions $g_{i,j}(\z^{\bmu_i})$, $1\leq i \leq r$, $1\leq j \leq m_i$, are related by the $T$-Mahler system associated with the matrix $B(\z)$, which proves (d). Since, for every $i$, $\bbeta^{\bmu_i}$ is regular w.r.t.\ the $q^\ell$-Mahler system associated with the matrix 
	$A_i(z)$, the point $\bbeta$ is regular  w.r.t.\ the $T$-Mahler system with matrix $B(\z)$. Furthermore, since the coordinates of $\bbeta$ are multiplicatively independent and of modulus smaller that $1$, it follows from \cite[Theorem~5.9]{AF24_Annals} that $(T,\bbeta)$ is admissible, hence (e) is satisfied. 
	Since $T=q^\ell{\rm I}_t$, (f) also holds true. 
	Finally, since the numbers $\alpha_1,\ldots,\alpha_r$ are pairwise multiplicatively independent, so are the numbers $\bbeta^{\bmu_1},\ldots,\bbeta^{\bmu_r}$. Thus, the vectors $\bmu_1,\ldots,\bmu_r$ are pairwise linearly independent over $\mathbb Q$, which proves (g). 
\end{proof}

\subsection{Proof of Theorem \ref{thm: mainbis}}

We are now ready to prove our main result. We assume that none of the complex numbers $f_1(\alpha_1),\ldots,f_r(\alpha_r)$ belongs to $\mathbb K$, so that it remains to prove that they are algebraically independent over $\Q$.  We first notice that, according to \cite[Corollaire 1.8]{AF1}, this assumption implies that  the numbers $f_1(\alpha_1),\ldots,f_r(\alpha_r)$ are all transcendental.  

Let us divide the natural numbers $1,\ldots,r$ into $s$ classes $\mathcal I_1,\ldots,\mathcal I_s$, such that 
$i$ and $j$ belong to the same class if and only if  $q_i$ and $q_j$ are multiplicatively dependent. 
Since an $M_q$-function  is also an $M_{q^k}$-function for every integer $k\geq 1$, we can assume without any loss of generality that $q_i=q_j:=\rho_k$ whenever $i$ and $j$ belong to the same class $\mathcal I_k$.   
Set $\mathcal E := \{f_1(\alpha_1),\ldots,f_r(\alpha_r)\}$ and 
$\mathcal E_k:=\{f_{i}(\alpha_{i})\, : \, i \in \mathcal I_k\}$, $1\leq k \leq s$.

\medskip 
For each $k \in \{1,\ldots,s\}$, we consider the Mahler system given by Proposition~\ref{prop:existenceMpoints} when applied with $q=\rho_k$ and with the pairs 
$(f_i(z),\alpha_i)$, $i \in \mathcal I_k$. Let $\bbeta_k, (\bmu_i)_{i \in \I_k}, T_k, \z_k, (g_{i,j}(z))_{i\in \mathcal I_k,1\leq j \leq m_i}$ and $B_k(\z_k)$ denote, respectively, the corresponding algebraic point, vectors of nonnegative integers, transformation, vector of indeterminates, family of $M_{\rho_k}$-functions and matrix associated with the corresponding $T_k$-Mahler system. Proposition~\ref{prop:existenceMpoints} ensures that each pair $(T_k,\bbeta_k)$ is admissible and that the point $\bbeta_k$ is regular w.r.t. the $T_k$-Mahler system associated with the matrix $B_k(\z_k)$.  Since the numbers 
$\rho_1,\ldots,\rho_s$ are pairwise multiplicatively independent, Condition (f) of Proposition~\ref{prop:existenceMpoints} further implies that the spectral radii of $T_1,\ldots,T_s$ 
are pairwise multiplicatively independent. Thus, we can apply \cite[Corollary 3.9]{AF24_Annals} to these $s$ Mahler systems. 
We deduce that 
\begin{equation}\label{eq:Es}
{\rm tr.deg}_{\Q}(\mathcal E) = \sum_{k=1}^s {\rm tr.deg}_{\Q}(\mathcal E_k)\, .
\end{equation}
Now, let us fix $k\in\{1,\ldots,s\}$ and set $\mathcal F_k:=\{g_{i,j}(\bbeta_k^{\bmu_i})\, : \, i\in \mathcal I_k,\,1\leq j \leq m_i\}$ and 
$$
 \mathcal F_{k,i}:=\{(g_{i,j}(\bbeta_k^{\bmu_i})\,:\,1\leq j \leq m_i\},\quad i \in \mathcal I_k\,.
$$
Applying \cite[Corollary 3.5]{AF24_Annals} to the $T_k$-Mahler system associated with the matrix $B_k(\z_k)$, we obtain that
$$
{\rm tr.deg}_{\Q}(\mathcal F_k) = {\rm tr.deg}_{\Q(\z_k)}(g_{i,j}(\z_k^{\bmu_i})\,: \, i \in \mathcal I_k,\,1\leq j \leq m_i)\,.
$$
Since Condition (g) of Proposition~\ref{prop:existenceMpoints} ensures that the vectors $\bmu_i$, $i \in \I_k$, are pairwise linearly independent over $\mathbb Q$, it follows from Lemma~\ref{lem:purity_functions} that
\begin{multline*}
 {\rm tr.deg}_{\Q(\z_k)}(g_{i,j}(\z_k^{\bmu_i})\,: \, i \in \mathcal I_k,\,1\leq j \leq m_i) \\=
 \sum_{i \in \mathcal I_k}  {\rm tr.deg}_{\Q(z)}(g_{i,j}(z)\,: \, 1\leq j \leq m_i)\,.
\end{multline*}
For each $i \in \mathcal I_k$, we infer from Condition (c) of Proposition~\ref{prop:existenceMpoints} that we can apply \cite[Corollary~3.5]{AF24_Annals} to the Mahler system connecting $g_{i,1}(z),\ldots,g_{i,m_i}(z)$ at the regular 
point $\bbeta_k^{\bmu_i}$. We obtain that 
$$
{\rm tr.deg}_{\Q(z)}(g_{i,j}(z)\,: \, 1\leq j \leq m_i) = {\rm tr.deg}_{\Q}(\mathcal F_{k,i})\,.
$$
Combining these three identities, we get that 
\begin{equation}\label{eq: transdsum}
{\rm tr.deg}_{\Q}(\mathcal F_k) = \sum_{i \in \mathcal I_k}  {\rm tr.deg}_{\Q}(\mathcal F_{k,i})\,.
\end{equation}
We infer from Condition (b)  of Proposition~\ref{prop:existenceMpoints} that $f_i(\alpha_i) \in\mathcal F_{k,i}$, so that 
$\mathcal F_k=\cup_{i\in\mathcal I_k} F_{k,i}$ and $\mathcal E_k=\cup_{i\in\mathcal I_k} f_i(\alpha_i)$. 
Then, it follows from \cite[Lemma~10.3]{AF24_Annals}  and \eqref{eq: transdsum} that 
$$
{\rm tr.deg}_{\Q}(\mathcal E_k) = \sum_{i \in \mathcal I_k}  {\rm tr.deg}_{\Q}(f_i(\alpha_i))\,.
$$
Since $f_i(\alpha_i)$ is transcendental for all $i$, we have ${\rm tr.deg}_{\Q}(f_i(\alpha_i))=1$ and we deduce that ${\rm tr.deg}_{\Q}(\mathcal E_k) ={\rm Card}(\mathcal I_k)$. Then, it follows from \eqref{eq:Es} that
$$
{\rm tr.deg}_{\Q}(\mathcal E)=\sum_{k=1}^s {\rm Card}(\mathcal I_k)= r\,.
$$
Hence the numbers $f_1(\alpha_1),\ldots,f_r(\alpha_r)$ are algebraically independent over $\Q$, just as we wanted. \qed

\end{document}